\documentclass[smallextended,numbook,runningheads]{svjour3}     
\smartqed  
\usepackage{graphicx}
\usepackage{amsmath}
\usepackage{epstopdf} 
\usepackage{mathptmx}      
\usepackage{amsfonts,amsmath,mathrsfs}
\newcommand{\e}{\mathbf{e}}
\newcommand{\fac}{{\rm !}}
\newcommand{\x}{\mathbf{x}}

\newcommand{\y}{\mathbf{y}}

\newcommand{\R}{\mathbb{R}}

\newcommand{\N}{\mathbb{N}}

\newcommand{\C}{\mathbb{C}}

\newcommand{\A}{\mathbf{A}}

\newcommand{\z}{\mathbf{z}}

\newcommand{\by}{\mathbf{y}}

\newcommand{\bx}{\mathbf{x}}

\newcommand{\ba}{\mathbf{a}}

\newcommand{\bc}{\mathbf{c}}

\newcommand{\bell}{\boldsymbol{\ell}}
\newcommand{\bxi}{\boldsymbol{\xi}}
\newcommand{\bpsi}{\boldsymbol{\psi}}

\newcommand{\om}{\boldsymbol{\Omega}}
\journalname{BIT}
\begin{document}

\title{Simple formula for  integration of polynomials on a simplex
\thanks{Research funded by the European Research Council (ERC) under the European's Union Horizon 2020 research and innovation program (grant agreement 666981 TAMING)}}

\subtitle{Simple formula for integration}


\author{Jean B. Lasserre}


\institute{J.B. Lasserre \at
              LAAS-CNRS and Institute of Mathematics, University of Toulouse, France \\
              Tel.: +33-5-61336415\\
              Fax: +33-5-61336936\\
              \email{lasserre@laas.fr} }          

\date{Received: date / Accepted: date}

\maketitle

\begin{abstract}
We show that integrating a polynomial $f$ of degree $t$ on an arbitrary simplex (with respect to Lebesgue measure)
reduces to evaluating $t$ homogeneous related Bombieri polynomials of degree $j=1,2,\ldots,t$, each at a unique point 
$\bxi_j$ of the simplex. This new and very simple formula 
could be exploited in finite (and extended finite) element methods, as well as in applications where such integrals 
must be evaluated. A similar result also holds for a certain class of positively homogeneous functions that are integrable on the canonical simplex.
\end{abstract}

\section{Introduction}

We consider the  problem of integrating a polynomial $f\in\R[\bx]$ 
on an arbitrary simplex of $\R^n$ and with respect to the Lebesgue measure.
After an affine transformation this problem is completely equivalent 
to integrating a related polynomial of same degree on the canonical simplex 
$\boldsymbol{\Delta}=\{\bx\in\R^n_+:\e^T\bx\leq 1\}$ where $\e=(1,\ldots,1)\in\R^n$.
Therefore the result is first proved on $\boldsymbol{\Delta}$ and then  transferred back
to the original simplex. The result is also extended to 
positively homogeneous functions of the form $\sum_{\alpha\in M} f_\alpha x_1^{\alpha_1}\cdots x_n^{\alpha_n}$
where $M\subset\R^n$ is finite and $-1<\alpha_i\in \R,\:\forall i$. \\
 
 \subsection{Background}
 
In addition to being a mathematical problem of independent interest, 
integrating a polynomial on a polytope has important applications
in e.g. computational geometry,  in approximation theory (constructing splines), in finite element methods, 
to cite a few; see e.g. the discussion in Baldoni et al. \cite{baldoni}. 
In particular, because of applications 
in finite (and extended finite) element methods and also for volume computation
in the Natural Element Method (NEM), there has been a recent renewal of interest in
developing efficient integration numerical schemes for polynomials on convex and non-convex polytopes

For instance the HNI (Homogeneous Numerical Integration)
technique developed in Chin et al. \cite{chin} and based on \cite{lass-poly}, has been proved to be particular efficient in
some finite (and extended finite) element methods; see e.g. Antonietti et al. \cite{galerkin},
Chin and Sukumar \cite{suku}, Nagy and Benson \cite{nagy} for intensive experimentation, Zhang et al. \cite{zhang} for NEM,
Frenning \cite{frenning} for DEM (Discrete Element Method), Leoni and Shokef \cite{leoni} for volume computation. 
For exact volume computation of polytopes in computational geometry, the interested reader is also referred to B\"ueler et al. \cite{bueler} 
and references therein.

For integrating a polynomial on a polytope, one possible route is to
use the HNI method developed in \cite{lass-poly,chin} and also extended in \cite{galerkin},
without partitioning the polytope in simplices. Another direction is to consider
efficient numerical schemes for {\em simplices} since quoting Baldoni et al. \cite{baldoni} 
``{\em among all polytopes, simplices are the fundamental case to consider for integration since any convex
polytope can be triangulated into finitely many simplices}". 
In  \cite{baldoni} the authors analyze the computational complexity of the latter case and 
describe several formulas; in particular they show that the problem is NP-hard in general.
However, if the number of variables is fixed then one may integrate polynomials of linear forms
efficiently (with Straight-Line program for evaluation) and if the degree is fixed one may 
integrate any polynomial  efficiently. They also describe several formulas in closed form
for integrating powers of linear forms \cite[Corollary 12]{baldoni} and also
arbitrary homogeneous polynomials \cite[Proposition 18]{baldoni} and \cite{lass-avra},
all stated in terms of a summation over vertices of the simplex.

\subsection{Our main result} 

With $f\in\R[\x]$, $\x\mapsto f(\x)=\sum_\alpha f_\alpha\,\x^\alpha$ is associated the Bombieri-type polynomial
$\hat{f}\in\R[\x]$:
\begin{equation}
\label{f-hat}
\x\mapsto \hat{f}(\x)\,=\,\sum_{\alpha\in\N^n}\hat{f}_\alpha\,\x^\alpha\,=\,\sum_{\alpha\in\N^n}\,
\alpha_1\fac\cdots\alpha_n\fac\,f_\alpha\,\x^\alpha,\qquad\x\in\R^n.\end{equation}

We establish the following simple formula:
\begin{theorem}
\label{th-main}
Let $\boldsymbol{\Delta}\subset\R^n$ be the canonical simplex.
If $f\in\R[\x]$ is a polynomial of total degree $t\in\N$ then with $\bxi_j:=\e/((n+1)\cdots (n+j))^{1/j}$:
\begin{equation}
\label{th-main-1}
\int_{\boldsymbol{\Delta}}f\,d\x\,=\,\frac{1}{n{\rm !}}\:(\hat{f}_0+\sum_{j=1}^t
\hat{f}_j(\bxi_{j})),
\end{equation}
where $f=\sum_{j=0}^t f_j$ and each $f_j$ is a homogeneous form\footnote{$f_j$ is the unique form of degree $j$ which is the sum of all monomials of degree $j$ of $f$ (with their coefficient).}  of degree $j$. 

Similarly, let $M\subset\R^n$ be a finite set and let $f:(0,+\infty)^n\to\R$ with
\begin{equation}
\label{def-pos-homog}
\x\mapsto f(\x)\,:=\,\sum_{\alpha\in M} f_\alpha\,\x^\alpha,\quad\alpha\in\R^n\,,\quad \alpha_i>-1,\:i=1,\ldots,n,\end{equation}
be positively homogeneous of degree $t\in\R$ (i.e. $\sum_i\alpha_i=t$ for all $\alpha\in M$). Then
\begin{equation}
\label{th-main-2}
\int_{\boldsymbol{\Delta}}f\,d\x\,=\,\frac{1}{n{\rm !}}\:\hat{f}(\bxi_t),
\end{equation}
where $\bxi_t:=\frac{1}{\theta}\,\e\,\in\boldsymbol{\Delta}$ with $\theta^t=\Gamma(1+n+t)/\Gamma(1+n)$, and
\[\x\mapsto \hat{f}(\x)\,:=\,\sum_{\alpha\in\,M} \Gamma(1+\alpha_1)\cdots\Gamma(1+\alpha_n)\,f_\alpha\,\x^\alpha.\]
\end{theorem}
Theorem \ref{th-main} states that integrating a polynomial $f$ of degree $t$ on the canonical simplex 
$\boldsymbol{\Delta}$ can be done by evaluating each Bombieri form $\hat{f}_j$ at a unique point $\bxi_j\in\boldsymbol{\Delta}$.
In addition all points $(\bxi_j)$ are on a line between the origin and the point $\e/n$
on the boundary of $\boldsymbol{\Delta}$. To the best of our knowledge, and despite their simplicity,
we have not been able to find formula \eqref{th-main-1} or  \eqref{th-main-2} in the literature, even if they can be 
obtained in several relatively straightforward manners from previous results in the literature. 

\emph{The point that we make in \eqref{th-main-1} is to relate $\int_{\boldsymbol{\Delta}}\,f\,d\x$ to
point evaluation of its Bombieri-polynomial at only $t$ very specific points of the simplex $\boldsymbol{\Delta}$; and similarly for positively homogeneous functions of the form \eqref{def-pos-homog}.}

Similarly, integrating a polynomial $f$ of degree $t$ on an arbitrary
simplex $\om\subset\R^n$ can be done by evaluating related polynomials $h_j$ of degree $j$,
each at a certain point of $\om$. Indeed an arbitrary (full-dimensional) simplex 
 can be mapped to the canonical simplex $\boldsymbol{\Delta}$ by an affine  transformation $\by=\A\,(\x-\ba)$ for some
 real nonsingular  matrix $\A$ and vector $\ba\in\R^n$. Therefore \eqref{th-main-1} translates
 to a similar formula on $\om$ with ad-hoc polynomials and $t$ aligned points $\x_j\in\om$.
  
 Hence formula \eqref{th-main-1} is much simpler than those in \cite{baldoni,lass-avra}. In particular 
 it is valid for an arbitrary polynomial and there are only  $t$ points involved if the degree 
 of the polynomial is $t$. In contrast, for integrating a $t$-power of a linear form,
 the formula in  \cite{baldoni} requires a summation at the $(n+1)$ vertices and 
in \cite{lass-avra}, integrating a form of degree $t$ requires ${n+t\choose n}$ evaluations 
of a multilinear form at the $(n+1)$ vertices. 
 
 We would like to emphasize that \eqref{th-main-1} resembles  a cubature formula but is {\em not}. A cubature
formula is of the form:
\begin{equation}
\label{cubature}
\int_{\boldsymbol{\Delta}} f\,d\x\,=\,\sum_{j=1}^s w_j\,f(\x_j),\quad \forall f\in\R[\x]_t\,,\end{equation}
for some integer $s$ and points $(\x_j)\subset\boldsymbol{\Delta}$ with associated weights $w_j$, $j=1,\ldots,s$.
However Theorem \ref{th-main} suggests that as long as polynomials are concerned,
the simpler alternative formula \eqref{th-main-1} could be preferable 
to cubature formulas involving many points. The point of view is different. Instead of evaluating
the {\em single} polynomial $f$ of degree $t$ at several points $\x_j$ in \eqref{cubature}, in \eqref{th-main-1} one evaluates $t$ 
{\em other} polynomials $\hat{f}_j$ of degree $j$, $j=1,\ldots t$ (simply related to $f$); each polynomial
$\hat{f}_j$ is evaluated at a single point only.

Our technique of proof is relatively simple. It uses (i) Laplace transform technique and homogeneity
to reduce integration on $\boldsymbol{\Delta}$ with respect to (w.r.t.) Lebesgue measure to integration on the nonnegative orthant
w.r.t. exponential density; this technique was already advocated in \cite{lass-zeron} for computing certain
multivariate integrals and in \cite{volume} for volume computation of polytopes. Then (ii) integration of monomials w.r.t. exponential density can be done in closed-form
and results in a simple formula in closed form.

Interestingly and somehow related, recently Kozhasov et al. \cite{sturmfels} have considered
integration of a ``monomial" $\y^{\alpha-\e}$  (with $0<\alpha\in\R^n$) 
with respect to exponential density on the positive orthant.  In \cite{sturmfels} 
the density $\y\mapsto \prod_i \Gamma(\alpha_i)\,y_i^{\alpha_i-1}$
is called the {\em Riesz kernel} of the monomial $\x^{-\alpha}$. The Riesz kernel 
offers a certificate of positivity for a function to be completely monotone (a strong positivity property of functions on cones).

\section{Main result}

\subsection{Notation, definitions and preliminary result}
Let $\R[\x]$ denote the ring of real polynomials in the variable $\x=(x_1,\ldots,x_n)$.
With $\N$ the set of natural numbers, a polynomial $f\in\R[\x]$ is written
\begin{equation}
\label{def-f}
\x\mapsto f(\x)=\sum_{\alpha\in\N^n}f_\alpha\,\x^\alpha,\end{equation}
in the canonical basis of monomials. A polynomial $f$
is homogeneous of degree $t$ if $f(\lambda\,\x)=\lambda^t\,f(\x)$ for all $\lambda\in\R$ 
and all $\x\in\R^n$. For $\alpha\in\R^n$ let $\vert\alpha\vert:=\sum_i\vert\alpha_i\vert$.
Let $\R^n_+:=\{\x\in\R^n:\x\geq0\}$ denotes the positive orthant of $\R^n$.
A function $f:(\R_+\setminus\{0\})^n\to\R$ is positively homogeneous of degree $t\in\R$ if
\[f(\lambda\,\x)\,=\,\lambda^t\,f(\x),\quad \forall \x\in (\R_+\setminus\{0\})^n,\:\forall\lambda>0,\]
and a polynomial $f$ is homogeneous of degree $t\in\R$ if
\[f(\lambda\,\x)\,=\,\lambda^t\,f(\x),\quad \forall \lambda\in\R,\:\forall \x\in\R^n.\]

Denote by $\boldsymbol{\Delta} \subset\R^n_+$ the canonical simplex $\{\x\in\R^n_+: \e^T\x\leq 1\}$ where
$\e=(1,\ldots,1)\in\R^n$. For $0<\x\in\R^n$ let $1/\x=(\frac{1}{x_1},\ldots,\frac{1}{x_n})\in\R^n$.

With a polynomial $f\in\R[\x]$ in \eqref{def-f}
is associated the ``Bombieri" polynomial:
\begin{equation}
\label{hat}
\x\mapsto \hat{f}(\x)\,=\,\sum_{\alpha\in\N^n}\hat{f}_\alpha\,\x^\alpha\,=\,\sum_{\alpha\in\N^n}\,f_\alpha\,
\alpha_1\fac\cdots\alpha_n\fac\,\x^\alpha,\qquad\x\in\R^n.
\end{equation}

\subsection{Main result}~

After an affine transformation, integrating $f$ on an arbitrary full-dimensional simplex $\om\subset\R^n$
reduces to integrating a related polynomial of same degree on the canonical simplex 
$\boldsymbol{\Delta}=\{\x\in\R^n_+:\e^T\x\leq 1\}$ where $\e=(1,\ldots,1)\in\R^n$. Therefore in this section we consider integrals
of polynomials (and a certain type of positively homogeneous functions) on the canonical simplex $\boldsymbol{\Delta}$.

\begin{theorem}
\label{th1}
Let $0<\z\in\R^n$ and let $\boldsymbol{\boldsymbol{\boldsymbol{\Delta}}}_\z=\{\x\in\R^n_+:\:\z^T\x\leq1\}$.
Let  $f$ be a positively homogeneous function of degree $t\in\R$ on $(\R_+\setminus\{0\})^n$ 
such that $t>-(1+n)$ and $\int_{\boldsymbol{\boldsymbol{\Delta}}_\z}\vert f\vert \,d\x <\infty$. Then:
\begin{equation}
\label{th1-1}
\displaystyle\int_{\boldsymbol{\boldsymbol{\Delta}}_\z} f(\x)\,d\x\,=\,\frac{1}{\Gamma(1+n+t)}\,\displaystyle\int_{\R^n_+} f(\x)\,\exp(-\z^T\x)\,d\x.
\end{equation}
If $f$ is a homogeneous polynomial of degree $t$ (hence $t\in\N$) then
\begin{equation}
\label{th1-2}
\displaystyle\int_{\boldsymbol{\boldsymbol{\Delta}}_\z} f(\x)\,d\x\,=\,\frac{1}{(n+t){\rm !}}\,\hat{f}(1/\z)\,\frac{1}{\z^\e},
\end{equation}
and in particular with $\z=\e$:
\begin{equation}
\label{th1-3}
\displaystyle\int_{\boldsymbol{\Delta}} f(\x)\,d\x\,=\,\frac{1}{(n+t){\rm !}}\,\hat{f}(\e)\,=\,\frac{1}{n{\rm !}}\,\hat{f}(\bxi_{t}),
\end{equation}
where $\bxi_{t}=\frac{1}{\theta}\,\e\in\boldsymbol{\Delta}$ and $\theta^t=(n+1)\cdots (n+t)$.
\end{theorem}
\begin{proof}
Let $0<\z\in\R^n$ be fixed and let $h:\R\to\R$ be the function:
\[y\mapsto h(y)\,:=\,\int_{\{\x\in\R^n_+: \z^T\x\leq y\}}f(\x)\,d\x.\]
Observe that $h=0$ on $(-\infty,0]$ and in addition, $h$ is positively homogeneous of degree $n+t$,
so that $h(y)=y^{n+t}\,h(1)$ (well defined since $h(1)$ is finite). As $t>-(n+1)$,
its Laplace transform $H:\C\to\C$ is well defined and reads:
\begin{equation}
\label{laplace}
\lambda\mapsto H(\lambda)\,=\,\frac{\Gamma(1+n+t)}{\lambda^{n+t+1}}\,h(1),\quad \lambda\in\C,\:\Re(\lambda)>0.\end{equation}
On the other hand, for real $\lambda>0$:
\begin{eqnarray*}
H(\lambda)&=&\int_0^\infty h(y)\,\exp(-\lambda\,y)\,dy\\
&=&\int_0^\infty \exp(-\lambda\,y)\,\left(\int_{\{\x\in\R^n_+:\z^T\x\leq y\}}f(\x)\,d\x\right)\,dy\\
&=&\int_{\R^n_+}f(\x)\left(\int_{\z^T\x}^\infty \exp(-\lambda\,y)\,dy)\right)\,d\x\quad\mbox{[by Fubini-Tonelli]}\\
&=&\frac{1}{\lambda}\int_{\R^n_+}f(\x)\,\exp(-\lambda\,\z^T\x)\,d\x\\
&=&\frac{1}{\lambda^{n+t+1}}\int_{\R^n_+}f(\x)\,\exp(-\z^T\x)\,d\x\,.
\end{eqnarray*}
Identifying with \eqref{laplace} yields \eqref{th1-1}.
Next, to get \eqref{th1-2} observe that for $0<z\in\R$, $\int_0^\infty y^k\exp(-z\,y)\,dy=k\fac/z^{k+1}$ for all $k\in\N$, and therefore:
\[\int_{\R^n_+}\y^\alpha\,\exp(-\z^T\y)\,d\y\,=\,\frac{1}{\z^\e}\alpha_1\fac\cdots\alpha_n\fac\,(1/\z)^\alpha,\]
for every $\alpha\in\N^n$. Summing up yields the result \eqref{th1-2} and \eqref{th1-3} with $\z=\e$. Finally, the last equality
of \eqref{th1-3} is obtained by homogeneity of $\hat{f}$.
\end{proof}
As the reader can see, formula \eqref{th1-3} is extremely simple and only requires evaluating $\hat{f}$ at 
the unique point
$\bxi_t\in\boldsymbol{\Delta}$. This in contrast to the formula in \cite{lass-avra} which requires a sum
of ${n+t\choose n}$ terms, each involving evaluations at the vertices of $\boldsymbol{\Delta}$.\\

\begin{remark}
\label{rem-kernel}
Notice that \eqref{th1-1} can also be interpreted as follows: Let $f:\R^n_+\to\R_+$,
be positively homogeneous of degree $t\in\R$ and as in Theorem \ref{th1}.
Define  the function $h_f:\R^n_+\to\R_+$, by: $\z\mapsto 
\Gamma(1+n+t)\int_{\boldsymbol{\Delta}_\z}f(\x)\,d\x$. 
Then $h_f$ is the multidimensional Laplace transform of $f$, or equivalently in the terminology of Kozhasov et al. \cite{sturmfels},
$h_f$ is completely monotone\footnote{A real-valued function $f:(\R\setminus\{0\})^n\to\R$
is completely monotone if $(-1)^k\frac{\partial^k f}{\partial x_{i_1}\cdots\partial x_{i_k}}(\x)\geq0$ for all
$\x\in(\R\setminus\{0\})^n$ and for all index sequences $1\leq i_1\leq \cdots\leq i_k\leq n$
of arbitrary length $k$.}
 (by the Bernstein-Hausdorff-Widder-Choquet theorem; see \cite[Theorem 2.5]{sturmfels}). 
Next, let $C\subset\R^n$ be an open cone with dual cone $C^*$. If $p\in\R[\x]$, $s\in\R$, and
\[p^s(\x)\,=\,\int_{C^*}\exp(-\y^T\x)\,d\mu(\y),\quad \forall\x\in\,C,\]
for some Borel measure $\mu$ on $C^*$, then $\mu$ is called a Riesz measure. In addition if $\mu$ has a density $q$ with respect to Lebesgue measure on $C^*$ then $q$ is called the Riesz kernel of $p^s$; see \cite[p. 4]{sturmfels}. \hspace{0.5cm}$\clubsuit$
\end{remark}
Hence from Remark \ref{rem-kernel}, with $C=(\R\setminus\{0\})^n$ and $d\mu=fd\x$, we obtain:
\begin{proposition}
\label{prop-aux}
For every positively homogeneous $f:(\R_+\setminus\{0\})^n\to\R_+$ of total degree $t\in\R$ as in Theorem \ref{th1},  the function
$h_f:\R^n_+\to\R_+$: 
\begin{equation}
\z\mapsto \quad h_f(\z)\,:=\,\Gamma(1+n+t)\,\displaystyle \int_{\{\x\in\R^n_+: \z^T\x\leq 1\}}f(\x)\,d\x\,,
\end{equation}
is completely monotone. In addition if $f=\x^{\alpha-\e}$ with $0<\alpha\in\R^n$ then
$f$ is the Riesz kernel of the function $\x\mapsto \x^{-\alpha}\cdot\prod_i\Gamma(\alpha_i)$
on $(\R_+\setminus\{0\})^n$.
\end{proposition}
Proposition \ref{prop-aux} is in the spirit of \cite{sturmfels}[Proposition 2.7].
Next given $\bell,\x\in\R^n$ denote by $\bell\cdot\x$ the scalar product $\bell^T\x$ and by 
$\bell\otimes\x\in\R^n$ the vector $(\ell_1x_1,\ldots,\ell_nx_n)$. Finally let
$\mathrm{E}_t\in\R[\x]$ the homogeneous polynomial of degree $t$ with all coefficients equal to $1$,
that is, $\x\mapsto \mathrm{E}_t(\x):=\sum_{\vert\alpha\vert=t} \x^\alpha$.

As a consequence of Theorem \ref{th1} we obtain:
\begin{corollary}
\label{cor1}
(i) Let $f$ be polynomial of degree $t$ and write $f=\sum_{j0}^t f_j$ where each $f_j$ is homogeneous of degree $j$.
Then 
\begin{equation}
\label{th-final-1}
\int_{\boldsymbol{\Delta}} f(\y)\,d\y\,=\,{\rm vol}(\boldsymbol{\Delta})\,(\hat{f}_0+\sum_{j=1}^t \hat{f}_j(\bxi_j)\,),
\end{equation}
where $\bxi_j=\e/((n+1)\cdots (n+j))^{1/j}$ and $\hat{f}_j$ is as in \eqref{hat}.

(ii) For every $\bell\in\R^n$ and $t\in\N$:
\begin{equation}
\label{th-final-2}
\int_{\boldsymbol{\Delta}} (\bell\cdot\x)^t\,d\x\,=\,\frac{1}{(n+t){\rm !}}\hat{f}(\e)\,=\,\frac{t{\rm !}}{(n+t){\rm !}}\,\mathrm{E}_t(\bell).
\end{equation}
\end{corollary}
\begin{proof}
(i) As $\int_{\boldsymbol{\Delta}} f(\y)\,d\y=\sum_{j=0}^t \int_{\boldsymbol{\Delta}} f_j(\y)\,d\y$, use Theorem \ref{th1} for each $f_j$ and sum up.
Next, (ii) 
is a direct consequence of Theorem \ref{th1} and the fact that (using
$(\bell\cdot\x)^t\,=\,t{\rm !}\sum_{\vert\alpha\vert=t}\bell^\alpha\x^\alpha/({\alpha_1}{\rm !}\cdots{\alpha_n}{\rm !})$),
\[\hat{f}(\x)\,=\,\widehat{(\bell\cdot\x)^t}\,=\,t{\rm !}\sum_{\vert\alpha\vert=t}\bell^\alpha\,\x^\alpha\,=\,
t{\rm !}\,\sum_{\vert\alpha\vert=t}\,(\bell\otimes\x)^\alpha\,=\,t{\rm !}\,\mathrm{E}_t(\bell\otimes\x),\]
and therefore $\hat{f}(\e)\,=\,t{\rm !}\,\mathrm{E}_t(\bell\otimes\e)=t{\rm !}\,\mathrm{E}_t(\bell)$.
\end{proof}
Hence Corollary \ref{cor1} states that integrating $f$ on $\boldsymbol{\Delta}$ reduces to evaluate 
each $\hat{f}_j$ at the unique point $\bxi_j\in\boldsymbol{\Delta}$, and sum up.
In addition, all points $\bxi_j$, $j=1,\ldots,t$, are aligned in $\boldsymbol{\Delta}$; they are between the origin $0$ and the point $\e/n\in\boldsymbol{\Delta}$,
on the line joining $0$ to $\e/n$. 
Again formula \eqref{th-final-1} and \eqref{th-final-2} are extremely
simple. The former only requires evaluating $\hat{f}_j$ at $\bxi_j$ ($t$ evaluations) and the latter only requires
evaluating the polynomial $\mathrm{E}_t$ at the point $\bell\in\R^n$. This is contrast with \cite[Corollary 12, p. 307]{baldoni} 
which is more complicated (even for a single form $\ell$).\\

Finally, Theorem \ref{th1} can be extended to a class of homogeneous functions
\begin{proposition}\label{prop1}
Let $M\subset\R^n$ be a finite set of indices and let $f:(\R_+\setminus\{0\})^n\to\R$ be positively homogeneous of degree $t\in\R$ and of the form
\[\x\mapsto \sum_{\alpha\in\,M} f_\alpha\,\x^\alpha,\quad \alpha\in\R^n,\:\alpha_i>-1,\quad\forall i=1,\ldots,n,\]
for some real coefficients $(f_\alpha)_{\alpha\in\,M}$. Then
\begin{equation}
\label{prop1-1}
\displaystyle\int_{\boldsymbol{\Delta}} f(\x)\,d\x\,=\,\frac{1}{\Gamma(1+n+t)}\,\hat{f}(\e)\,=\,\frac{1}{n{\rm !}}\,\hat{f}(\bxi_{t}),
\end{equation}
where $\bxi_{t}=\frac{1}{\theta}\,\e\in\boldsymbol{\Delta}$ with $\theta^t=\Gamma(1+n+t)/\Gamma(1+n)$, and
\[\x\mapsto \hat{f}(\x)\,:=\,\sum_{\alpha\in\,M} \Gamma(1+\alpha_1)\cdots\Gamma(1+\alpha_n)\,f_\alpha\,\x^\alpha.\]
\end{proposition}
\begin{proof}
Let $-1<\alpha\in\R^n$ with $\vert\alpha\vert=t$ be fixed and let $0<\z\in\R^n$. By Theorem \ref{th1}:
\begin{eqnarray*}
\int_{\boldsymbol{\Delta}_\z}\x^\alpha\,d\x&=&\frac{1}{\Gamma(1+n+t)}\,\int_{\R^n_+}\x^\alpha\exp(-\z^T\x)\,d\x\\
&=&\frac{1}{\Gamma(1+n+t)}\,\z^{-\alpha-\e}\,\prod_{i=1}^n\Gamma(1+\alpha_i).
\end{eqnarray*}
Summing up over all $\alpha\in M$ yields
\[\int_{\boldsymbol{\Delta}_\z}f(\x)\,d\x\,=\,\frac{\z^{-\e}}{\Gamma(1+n+t)}\,\sum_{\alpha\in M} f_\alpha\,\z^{-\alpha}\,\prod_{i=1}^n\Gamma(1+\alpha_i)\,=\,\frac{\z^{-\e}}{\Gamma(1+n+t)}\,\hat{f}(\z^{-\e}),\]
and with $\z=\e$ one obtains \eqref{prop1-1}. Finally, the last equality in \eqref{prop1-1}
uses the positive homogeneity of $\hat{f}$ and $\Gamma(1+n)=n{\rm !}$ for all  $n\in\N$.
\end{proof}

\begin{example}
\label{ex1}
For illustration purpose consider the elementary two-dimensional example (i.e., $n=2$) where
$\boldsymbol{\Delta}=\{\x\geq0:\: x_1+x_2\leq 1\}$. 
With $\x\mapsto f(\x):=f_{10}\,x_1+f_{11}\,x_1x_2+f_{02}\,x_2^2$ one obtains
$\bxi_1=\e/3$ and $\bxi_2=\e/\sqrt{12}$. The right-hand-side of \eqref{th-final-1} reads:
\[\frac{1}{2}(f_{10}\,(\bxi_1)_1+f_{11}\,(\bxi_2)_1(\bxi_2)_2+ 2\,f_{02}\,(\bxi_2)^2_2)\,=\,
\frac{1}{2}(\frac{f_{10}}{3}+\frac{f_{11}}{12}+2\frac{f_{02}}{12}).\]
For instance with $\x\mapsto f(\x):=x_1+x_1x_2+x_2^2$, one obtains:
\[\frac{1}{2}((\bxi_1)_1+(\bxi_2)_1(\bxi_2)_2+ 2(\bxi_2)^2_2)\,=\,\frac{1}{2}(\frac{1}{3}+\frac{1}{12}+2\frac{1}{12})\,=\,\frac{7}{24},\]
and indeed $\int_{\boldsymbol{\Delta}} fd\x=7/24$. In 
Figure \ref{fig1} is displayed the $2D$-Simplex $\boldsymbol{\Delta}$ with the points $\bxi_1=\e/3$ and $\bxi_2=\e/\sqrt{12}$.
\begin{figure}
\includegraphics[width=.7\textwidth]{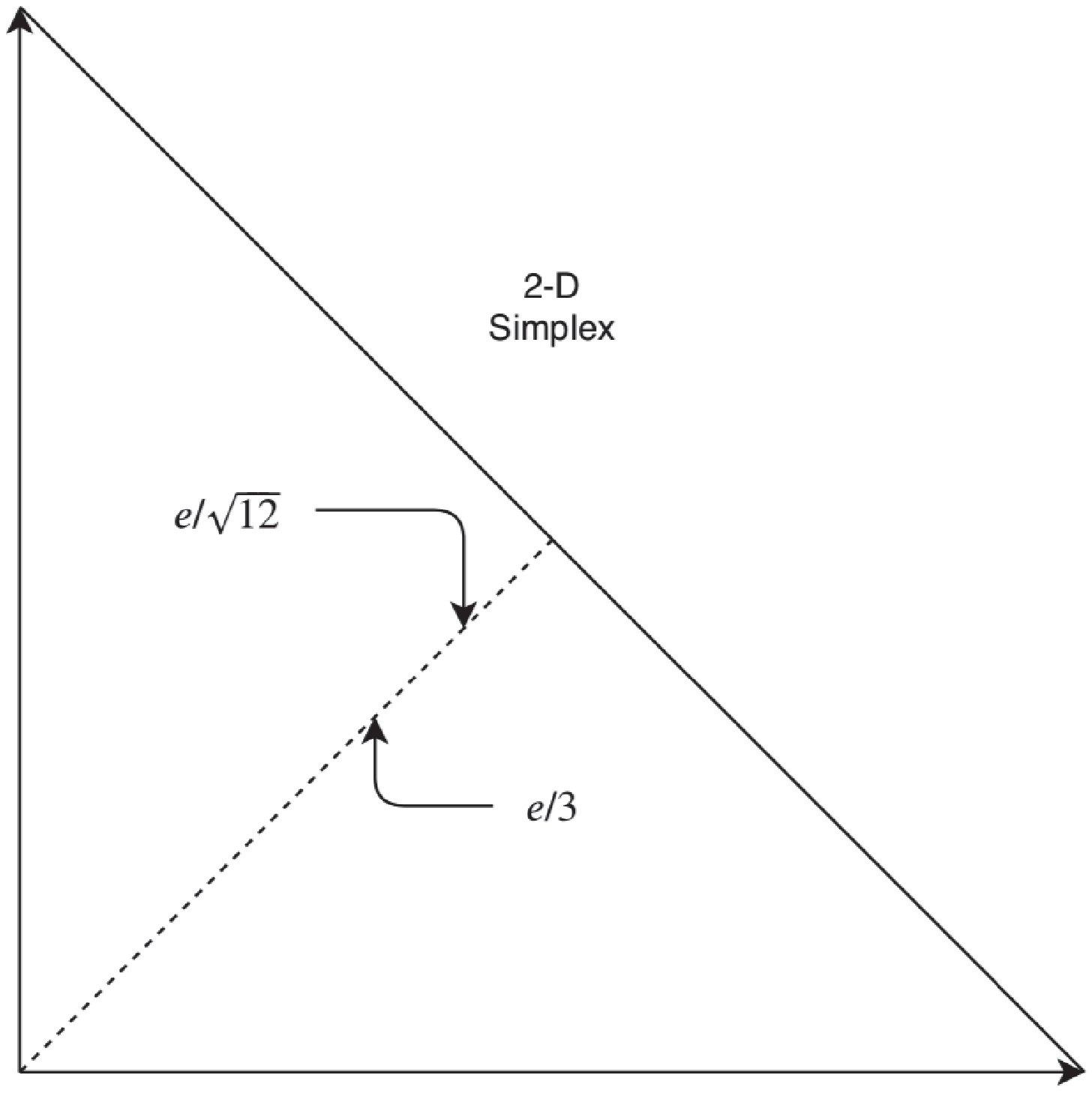}
\caption{$n=2$; Simplex $\boldsymbol{\Delta}$ and points $\bxi_1=\e/3$ and $\bxi_2=\e/\sqrt{12}$} 
\label{fig1}
\end{figure}
\end{example}

\subsection{Back to an arbitrary simplex}

Let $\om\subset\R^n$ be an arbitrary full-dimensional simplex (its $n$-dimensional Lebesgue volume is strictly positive). Then
$\om$ is mapped to $\boldsymbol{\Delta}$ by some affine transformation.
In a fixed basis, this is obtained by the change of variable $\y:=\A(\x-\ba)$ where
$\A$ is one non singular $n\times n$ real matrix and where $\ba\in\R^n$ is a vertex of $\om$. Then:
\[\x\in\om\quad\Leftrightarrow\quad \y:=\A\,(\x-\ba)\,\in\,\boldsymbol{\Delta}.\]
Next, if $f$ is a polynomial of degree $t$:
\begin{equation}
\label{back}
\int_{\om} f(\x)\,d\x\,=\,\frac{1}{{\rm det}(\A)}\,\int_{\boldsymbol{\Delta}} \underbrace{f(\A^{-1}\y+\ba)}_{g(\y)}\,d\y
\,=\,\frac{1}{{\rm det}(\A)}\,\int_{\boldsymbol{\Delta}} g(\y)\,d\y,\end{equation}
where $g$ has same degree as $f$.
Next, writing $g=\sum_{j=0}^t g_j(\y)$ with $g_j$ homogeneous of degree $j$, one has
\[\hat{g}_j(\y)\,=\,\hat{g}_j(\A(\x-\ba))\]
and therefore with $\bxi_j\in\boldsymbol{\Delta}$ as in Corollary \ref{cor1},
$\bpsi_j:=\A^{-1}\bxi_j+\ba\in\om$ and $\hat{g}_j(\bxi_j)=\hat{g}_j(\A(\bpsi_j-\ba))$. Therefore
combining \eqref{back} and \eqref{th-final-1} in Corollary \ref{cor1} yields:
\[\int_{\om} f(\x)\,d\x\,=\,\frac{1}{{\rm det}(\A)}\,\frac{1}{n{\rm !}}\,[\,\hat{g}_0+\sum_{j=1}^t\hat{g}_j(\A(\bpsi_j-\ba))\,],\]
the analogue for $\om$ of \eqref{th-final-1} for $\boldsymbol{\Delta}$.\\

In addition, let $f$ be homogeneous of degree $t$ with Waring-like decomposition\footnote{In number theory,
the Waring problem consists of writing any positive integer as a sum 
of a fixed number $g(n)$ of $n$th powers of integers, where $g(n)$ depends only on $n$. It generalizes to forms
as a generic form of degree $d$ can be written as a sum of $s$ $d$-powers of linear forms; $s$ is called the Waring rank of the form.}
\[\x\mapsto\quad\,f(\x)\,=\,\sum_{i=1}^s \varepsilon_i\,(\bc_i\cdot\x)^t,\]
 for finitely many $\bc_i\in\R^n$ and $\varepsilon_i\in\{-1,1\}$, $i=1,\ldots,s$.
Then letting  $\boldsymbol{\ell}_i=(\A^{-1})^T\bc_i$:
\begin{eqnarray*}
\int_{\om} f(\x)\,d\x&=&\frac{1}{{\rm det}(\A)}\,\sum_{i=1}^s\varepsilon_i\int_{\boldsymbol{\Delta}} (\bc_i\cdot \A^{-1}\y+\bc_i\cdot\ba)^t\,dt\\
&=&\sum_{i=1}^s\frac{\varepsilon_i}{{\rm det}(\A)}\,\int_{\boldsymbol{\Delta}} (\bc_i\cdot\ba+\boldsymbol{\ell}_i\cdot \y)^t \,dt\\
&=&\sum_{i=1}^s\frac{\varepsilon_i}{{\rm det}(\A)}\,\sum_{k=0}^t{t\choose k}(\bc_i\cdot\ba)^{t-k}\int_{\boldsymbol{\Delta}} (\boldsymbol{\ell}_i\cdot \y)^k \,dt\\
&=&\sum_{i=1}^s\frac{\varepsilon_i}{{\rm det}(\A)}\,\sum_{k=0}^t{t\choose k}(\bc_i\cdot\ba)^{t-k}\frac{k{\rm !}}{(n+k){\rm !}}
\mathrm{E}_k(\boldsymbol{\ell}_i)\\
&=&
\frac{1}{{\rm det}(\A)}\,\sum_{k=0}^t{t\choose k}\frac{k{\rm !}}{(n+k){\rm !}}\sum_{i=1}^s\varepsilon_i\,(\bc_i\cdot\ba)^{t-k}
\mathrm{E}_k(\boldsymbol{\ell}_i),
\end{eqnarray*}
where we have used Newton binomial formula and Corollary \ref{cor1}(ii) (and recall that $\mathrm{E}_k(\x)=\sum_{\vert\alpha\vert=k}\x^\alpha$).  
In particular, if $\ba=0$ and now with $\boldsymbol{\ell_i}:=(\A^{-1})^T\bc_i$, $i=1,\ldots,s$, it simplifies to
\[\int_{\om} f(\x)\,d\x\,=\,\frac{1}{{\rm det}(\A)}\,\frac{t{\rm !}}{(n+t){\rm !}}\sum_{i=1}^s\varepsilon_i\,\mathrm{E}_t(\boldsymbol{\ell}_i).\]

\section{Conclusion}
We have provided a very simple closed-form expression for the integral of an arbitrary  polynomial $f$
on an arbitrary full-dimensional simplex. Remarkably if $f$ has degree $t$, it consists of evaluating
$t$ polynomials (related to $f$) of degree $1,2,\ldots,t$, respectively, each at a unique point of the simplex.
To the best of our knowledge this simple formula is new and potentially useful in all applications
where such integrals need to be computed; for instance in finite and extended finite element methods.
Therefore it could provide a valuable addition to the arsenal of techniques already available 
for multivariate integration on polytopes.

\end{document}